\documentclass[11pt,reqno]{amsart}
\usepackage[vmargin=2cm,hmargin=2cm]{geometry}
\usepackage{amssymb, latexsym, amsfonts,amsmath,amsthm,color,etaremune,enumerate,tikz}
\usepackage{caption,float,textcomp,units,xfrac,url}
\allowdisplaybreaks

\title{Equidistribution and inequalities for partitions into powers}
\author{Alexandru Ciolan}
\address{Max-Planck-Institut f\"ur Mathematik, Vivatsgasse 7, 53111 Bonn, Germany}
\email{ciolan@mpim-bonn.mpg.de}
\newtheorem{Thm}{Theorem}

\newtheorem{Lem}{Lemma}

\newcommand{\cal}{\mathcal}

\newcommand{\bb}{\mathbb}

\theoremstyle{remark}

\theoremstyle{definition}
\newtheorem{Con}{Conjecture}
\newtheorem{Quest}{Question}
\newtheorem*{xcom}{Commentary}
\newtheorem{rem}{Remark}

\DeclareMathOperator{\Arg}{Arg}

\DeclareMathOperator{\Log}{Log}
\makeatletter
\let\@@pmod\pmod
\DeclareRobustCommand{\pmod}{\@ifstar\@pmods\@@pmod}
\def\@pmods#1{\mkern4mu({\operator@font mod}\mkern 6mu#1)}
\makeatother

\begin{document}
\begin{abstract} If $ p_k(a,m,n) $ denotes the number of partitions of $n$ into $k$th powers with a number of parts that is congruent to $ a $ modulo $m,$  then $p_2(0,2,n)\sim p_2(1,2,n)$ and the sign of the difference $p_2(0,2,n)- p_k(1,2,n)$ alternates with the parity of $n,$  as proven by recent work of the author (2020).~In this paper, we place the  problem in a broader framework. By analytic arguments using the circle method and Gauss sums estimates, we show that the same results hold for  any $ k\ge2. $ By combinatorial arguments, we show that the sign of the difference $p_k(0,2,n)- p_k(1,2,n)$ depends on the parity of $n$ for a larger class of partitions.   
\end{abstract}
\maketitle

\section{Introduction and statement of results}\label{Intro}
\subsection{Motivation}
A \textit{partition} of a positive integer $ n $ is a non-increasing sequence (often written as a sum) of positive integers, called \textit{parts}, that add up to $ n. $ By $ p(n) $ we denote the number of partitions of $ n, $ and by convention we let $ p(0)=1. $ For example, $ p(4)=5 $ as the partitions of $ 4 $ are $ 4, $ $ 3+1, $ $ 2+2, $ $ 2+1+1, $ and $ 1+1+1+1, $ this being the case of \textit{unrestricted} partitions. One can consider, however, partitions with various conditions imposed on their parts, such as partitions with all parts being in a set $  S $ satisfying certain properties. If $S$ is any set (finite or infinite) of positive integers, we denote by $p_S(n)$ the number of partitions of $n$ into parts that all belong to the set $S.$ In the particular case when $S=\{n^k:n\in\bb N\}$ is the set of perfect $k$th powers,  with $k\in\bb N,$ we will use the shorthand notation $p_k(n)$ instead, and it is with such partitions that the current paper will mostly be concerned. Additionally, we let $p_S(m,n)$ denote the number of partitions of $n$  with exactly $m$ parts, all from $S,$ and $p_S(a,m,n)$ that of partitions of $n$ that have a numbers of parts congruent to $ a $ modulo $m$, all from $S.$ The quantities $p_k(m,n)$ and $ p_k(a,m,n) $ are defined in a similar fashion as explained above, in the special case when $S=\{n^k:n\in\bb N\}.$ 
Answering a conjecture formulated by Bringmann and Mahlburg \cite{BM}, the author  proved the following \cite{AC}. 
\begin{Thm}[{\cite{AC}}]\label{Conj1} As $ n\to\infty, $ we have	
$$p_2(0,2,n)\sim p_2(1,2,n)\sim\frac{p_2(n)}2$$
and
\begin{equation}\label{GenericInequality}\begin{cases}
p_2(0,2,n)>p_2(1,2,n) & \text{if~$ n $~is even,}\\
p_2(0,2,n)<p_2(1,2,n) & \text{if~$ n $ is odd}.
\end{cases}\end{equation}
\end{Thm}
The statement of Theorem \ref{Conj1}, which is about partitions into squares, raises the natural question whether the same type of result holds for partitions into higher powers, or, more generally, into parts that are certain polynomial functions. Also, one might wonder whether similar results hold for moduli $ m>2. $
\subsection{Historical background}\label{history-sec} The earliest result of which the author is aware in the literature goes back to 1876 and is due to Glaisher \cite{Glaisher}, who proved   that $p_1(0,2,n)-p_1(1,2,n)=(-1)^n p_{\rm o}(n),$ where $ p_{\rm o}(n) $ counts the partitions of $ n $ into odd parts without repetitions. To compute asymptotics for $p(n),$ Hardy and Ramanujan  designed the famous \textit{circle method}, a breakthrough of their times,  while Wright \cite{WrightIII} improved on their method and computed asymptotics for $p_k(n).$ Roth and Szekeres \cite{RothSzek} computed asymptotics for $ p_U(n) $ in the case when $U=\{u_n\}_{n\ge1}$ is a sequence of positive integers which is increasing for $n\ge n_0$ and which satisfies a few growth properties, see conditions (I)--(II) from \cite[p. 241]{RothSzek}, under the restriction that no repeated parts are allowed. Liardet and Thomas \cite{Liardet} computed $p_U(n)$ while removing this condition and allowing repetitions. An example of such a set $ U=\{u_n\} $ is given by $ u_n=f(n), $ where $ f$ is a polynomial which takes only integral values for $x\in\bb Z$ and has the property that  for 
every prime $p$ there exists an integer $x$ such that $p\nmid f(x).$ Certainly, $f(n)=n^k$ is such a polynomial, for any $k\in\bb N.$  Most recently, Zhou \cite{Zhou} proved that if $U=\{f(n)\}_{n\ge1},$  with $f:\bb N\to \bb N$  a polynomial function satisfying similar  conditions to those from \cite{RothSzek}, then $$p_U(a,m,n)\sim \frac{p_U(n)}{m}$$ holds uniformly  as $n\to \infty$ for all $ m=o\Big(n^{\frac1{2+2\deg(f)}}(\log n)^{-\frac12}\Big), $  for any $0\le a\le m-1.$ However,  the methods from \cite{Zhou} do not seem to help in proving inequalities such as \eqref{GenericInequality}.
\subsection{Statement of results}The purpose of this paper is two-fold.~First, we prove that Theorem \ref{Conj1} extends, indeed, to partitions into perfect $k$th powers, for any $k\ge2.$ Second, we want to see for what other, more general, types of partitions do inequalities of the form \eqref{IneqsPowers} hold true. For these purposes, we would like to give both an analytic and a combinatorial proof to Theorem \ref{ThmPowers}, since we believe that the two approaches, independent of one another, are instructive in their own right. 
\begin{Thm}\label{ThmPowers} 	
For any $ k\ge2 $ we have, as $ n\to\infty, $	
\begin{equation}\label{Equidistri}
p_k(0,2,n)\sim p_k(1,2,n)\sim\frac{p_k(n)}{2}\end{equation} 
and
\begin{equation}\label{IneqsPowers}
\begin{cases}
p_k(0,2,n)>p_k(1,2,n) & \text{if~$ n $~is even,}\\
p_k(0,2,n)<p_k(1,2,n) & \text{if~$ n $ is odd}.
\end{cases}\end{equation}
\end{Thm}
The analytic approach relies on Wright's modular transformations for partitions into $k$th powers \cite{WrightIII}, on a modification of Meinardus's Theorem on asymptotics of infinite product generating functions which combines the circle and the saddle-point method, and on estimates of exponential Gauss sums; in particular, we invoke a bound that was established by Banks and Shparlinski \cite{BS} with the (somewhat unexpected and surprising) help of the effective lower estimates on center density found by Cohn and Elkies \cite{CE} in their work on the sphere packing problem. \par We find the connection between our partition question and the sphere packing problem to be rather interesting, and it is also this precise step that allows for a generalization to $k\ge2$ of the argument given in \cite{AC} for dealing with the case $k=2.$  While the equidistribution statement follows as a particular case of Corollary 1.2 from \cite{Zhou}, our argument, which is independent from that in \cite{Zhou}, proves both the equidistribution and the inequalities at the same time.
\par The combinatorial approach simplifies our work greatly in establishing the inequalities \eqref{IneqsPowers}, but it is not of much help in proving equidistribution results, at least not in a more general framework. Nevertheless, it allows us to prove the following.
\begin{Thm}\label{MoreGeneralThm}
Let $ f:\bb N\to \bb N $  be an increasing function such that 
\begin{enumerate}[{\rm a)}]\item $f(1)=1;$
\item $f(n)$ is odd if $n$ is odd$;$
\item $f(2n)=2\alpha f(n),$ for any $n\in\bb N$ and any fixed  $\alpha\in \bb N$.
\end{enumerate} If $S=\{f(n)\}_{n\ge1},$ then 
\begin{equation}\label{IneqsMoreGeneral}
\begin{cases}
p_S(0,2,n)>p_S(1,2,n) & \text{if~$ n $~is even,}\\
p_S(0,2,n)<p_S(1,2,n) & \text{if~$ n $ is odd}.
\end{cases}\end{equation}
\end{Thm} 

As it is easy to see that the power functions $f_k(n)=n^k$ satisfy the conditions of Theorem \ref{MoreGeneralThm}, the inequalities \eqref{IneqsPowers} follow as an immediate consequence of the above result.

\subsection{Outline}
\par The paper is structured as follows. In Section \ref{Philosophy} we use generating functions to give a reformulation of our problem. In Section \ref{combiproof} we discuss the combinatorial approach and we give the proof of Theorem \ref{MoreGeneralThm}. In Section \ref{analyticstrategy} we present the strategy of the analytic proof and we discuss the similarities and differences with the  proof of the same result from \cite{AC} in the case $ k=2 $. This will also be done, throughout the paper, in the form of commentaries at the end of the relevant sections. We consider this to be for the benefit of the reader interested in comparing the present paper with \cite{AC}.  In Sections \ref{MainTerm} and \ref{ErrorTerm} we prove two estimates which, combined, will provide the analytic proof of Theorem \ref{ThmPowers}, given in Section \ref{Proof of Theorem 2}. In Section \ref{Open} we propose some open problems and future research directions.

\section{A reformulation}\label{Philosophy}
\subsection{Notation} 
Before proceeding any further, let us introduce some notation used in the sequel. By $ \zeta_n=e^{\frac{2\pi i}{n}} $ we will denote the standard primitive $ n $th root of unity.  For reasons of space, we will sometimes use $ \exp (z) $ instead of $ e^z. $  Whenever required to take logarithms or to extract   roots of complex numbers, we will use principal branches,
and the principal branch of the complex logarithm will be denoted by $ \Log $. The Vinogradov symbols $ o,$  $O $ and $ \ll $ are used throughout with their standard meaning. 
\subsection{Generating functions}\label{reformulation} It is well-known (see, for example, \cite[Ch. 1]{And}) that 
\begin{equation}\label{productpowers}
\prod_{n=1}^{\infty}(1-q^{ n^k} )^{-1}=\sum_{n=0}^{\infty}p_k(n)q^n,\end{equation} where, as usual, for $ \tau\in\bb H $ (the upper half-plane)  we set $q=e^{2\pi i\tau}. $ 
Letting
\[H_k(q)=\sum_{n=0}^{\infty}p_k(n)q^n,\quad H_k(w;q)=\sum_{m=0}^{\infty}\sum_{n=0}^{\infty}p_k(m,n)w^mq^n,\quad H_{k,a,m}(q)=\sum_{n=0}^{\infty}p_k(a,m,n)q^n, \]
it is not difficult to see, by the orthogonality relations for roots of unity, that

\begin{equation}\label{orthogonal}
H_{k,a,m}(q)=\frac 1m H_k(q)+\frac 1m \sum_{j=1}^{m-1}\zeta_m^{-aj}H_k(\zeta_m^j;q).
\end{equation}
Using, in turn, \eqref{orthogonal} and eq. (2.1.1) from \cite[p.~16]{And}, we obtain  
\[H_{k,0,2}(q)-H_{k,1,2}(q)=H_{k}(-1;q)=\prod_{n=1}^{\infty}\frac{1}{1+q^{n^k}},\]
and, on substituting $ q\mapsto -q, $ we have
\begin{equation}\label{FirstRewriteProduct}H_{k}(-1;-q)=\prod_{n=1}^{\infty}\frac{1}{1+(-q)^{n^2}}=\prod_{n=1}^{\infty}\frac{1}{( 1+q^{2^kn^k}) ( 1-q^{(2n+1)^k}) }=\prod_{n=1}^{\infty}\frac{( 1-q^{2^kn^k}) ^2}{( 1-q^{2^{k+1}n^k}) ( 1-q^{n^k}) }.\end{equation}
On noting now that
\[H_k(-1;-q)=H_{k,0,2}(-q)-H_{k,1,2}(-q)=\sum_{n=0}^{\infty}a_k(n)q^n,\] where 
\begin{equation*}\label{a2fromp2}
a_k(n)=\begin{cases}
p_k(0,2,n)-p_k(1,2,n) & \text{if~$ n $~is even;}\\
p_k(1,2,n)-p_k(0,2,n) & \text{if~$n$~is odd,}
\end{cases}
\end{equation*} we see that
proving the inequalities from Theorem \ref{ThmPowers} is equivalent to showing that the coefficients $ a_k(n)$ of the infinite product $H_k(-1;-q),$ expressed as a $ q$-series, are positive as $ n\to\infty. $ For simplicity, we will denote $G_k(q)=H_k(-1;-q).$
%
In order to extract more information about the coefficients $a_k(n),$ one might try to compute them asymptotically. Indeed, this is the motivation of our analytic approach, and this is what we are going to do in Sections \ref{analyticstrategy}--\ref{Proof of Theorem 2}.

\section{A combinatorial approach}\label{combiproof}
The attentive reader (and especially the reader familiar with partition generating functions) might have noticed that all the previous identities hold not only for partitions into parts from the set $\{n^k:n\in\bb N\}$ but, more generallly, for any set $S\subseteq\bb N$ of positive integers. Therefore, all steps in Section \ref{reformulation} can be redone with any set $S$ instead of the set of $k$th powers. In this regard, identity \eqref{productpowers} would become 
\[\prod_{n\in S}(1-q^{ n} )^{-1}=\sum_{n=0}^{\infty}p_S(n)q^n,\] 
and all the other identities following it, including the orthogonality relation \eqref{orthogonal},  would turn into their corresponding analogues. More precisely, for a given (infinite) set $S\subseteq\bb N,$ if we let
\begin{equation*} 
a_S(n)=\begin{cases} 
p_S(0,2,n)-p_S(1,2,n) & \text{if~$ n $~is even;}\\
p_S(1,2,n)-p_S(0,2,n) & \text{if~$n$~is odd,}
\end{cases} 
\end{equation*} and  \[H_S(w;q)=\sum_{m=0}^{\infty}\sum_{n=0}^{\infty}p_S(m,n)w^mq^n,\]
then we obtain \[H_S(-1;-q)=\prod_{n\in S}\frac1{1+(-q)^n}=\sum_{n=0}^{\infty}a_S(n)q^n. \]
In Section \ref{reformulation}, in the particular case when $S=\{n^k:n\in\bb N\},$ we expressed the above product as shown in \eqref{FirstRewriteProduct}.
However, we can express this product in another way. Indeed, if $S=\{f(n):n\in\mathbb N\}$ with $ f:\bb N\to\bb N $ as in Theorem \ref{MoreGeneralThm}, we have 
\begin{align}
\label{SecondRewriteProd}
H_S(-1;-q)&=\prod_{\ell\in S}\frac{1-(-q)^{\ell}}{1-(-q)^{2\ell}}=\prod_{n\ge1}\frac{1-(-q)^{f(n)}}{1-(-q)^{2f(n)}}\nonumber\\&=\prod_{n\ge1}\frac{1-(-q)^{f(2n)}}{1-q^{2f(n)}}\prod_{n\ge1}\big(1-(-q)^{f(2n-1)}\big)\nonumber\\&=\prod_{n\ge1}\frac{1-q^{f(2n)}}{1-q^{2f(n)}}\prod_{n\ge1}\big(1+q^{f(2n-1)}\big),
\end{align}where the last identity follows from the fact that $f$ satisfies properties b) and c) of Theorem \ref{MoreGeneralThm}. In light of c), we further have 
\begin{equation}\label{LiardetTrick}
H_S(-1;-q)=\prod_{n\ge1}\frac{1-q^{2\alpha f(n)}}{1-q^{2f(n)}}\prod_{n\ge1}\big(1+q^{f(2n-1)}\big).
\end{equation}
\begin{proof}[Proof of Theorem \ref{MoreGeneralThm}]Regarded as a series in $q^2,$ the first product in \eqref{LiardetTrick} counts partitions into parts from $S$ with at most $ \alpha $ possible repetitions, while the second product, regarded as a series in $q,$ counts partitions into parts from $T=\{f(2n-1):n\in\bb N\}$ with at most one possible repetition (i.e., without repeated parts). Therefore, we obtain
\begin{equation}\label{recurrence}
a_S(n)=\sum_{0\le k\le n}c_{S,\alpha}(k)d_{T,1}(n-k),\end{equation}
where $c_{S,\alpha}(k)  $ is the number of partitions of $k$ into parts from $S$ with at most $\alpha$ possible repetitions, while $d_{T,1}(k)$ is the number of partitions of $k$ into parts from $T$ with at most one possible repetition (that is, without repeated parts). It is now clear that $a_S(n)\ge0$ for any $n\ge1,$ proving the claim.  \end{proof} 
\begin{proof}[Combinatorial proof of Theorem \ref{ThmPowers}]
As the sets $ S_k=\{n^k:n\in\bb N\} $ satisfy the hypotheses of Theorem \ref{MoreGeneralThm} for any $ k\ge2 $ (and, in fact, for any $ k\ge1 $) the inequalities \eqref{IneqsPowers} follow as an easy consequence of Theorem \ref{MoreGeneralThm}. The asymptotics of both $ c_{S,\alpha}(k) $ and $ d_{T,1}(k) $ can be obtained as special cases of the work of Liardet and Thomas, see Theorem 14.2 in \cite{Liardet}. Since the asymptotics will explicitly follow, later in this paper, from our analytic proof of Theorem \ref{ThmPowers}, we leave it as an exercise to the interested reader to derive asymptotics for $ a_S(n) $ using \eqref{recurrence} and the results of \cite{Liardet}, and to thus prove the equidistribution.
\end{proof}
\section{Strategy of the analytic approach}\label{analyticstrategy}
Having discussed the combinatorial approach to Theorems \ref{ThmPowers} and \ref{MoreGeneralThm}, we will focus from now on solely on analytic aspects and on partitions into $k$th powers. For this, we recall the first representation of the product $G_k(q),$ given in \eqref{FirstRewriteProduct}, which says that 
\begin{equation*}
G_k(q)=\prod_{n=1}^{\infty}\frac{( 1-q^{2^kn^k}) ^2}{( 1-q^{2^{k+1}n^k}) ( 1-q^{n^k}) }=\sum_{n=0}^{\infty}a_k(n)q^n.\end{equation*} 
What we want to do, and what we will indeed do in what follows, is to compute asymptotics for the coefficients $ a_k(n) $ and to prove that they are positive.

\subsection{Meinardus's Theorem} The reader familiar with asymptotics of infinite product generating functions might have already recognized  the similarity between the infinite product expression for $G_k(q)$ and the one studied by Meinardus in \cite{Mein}, which, on writing $ q=e^{-\tau} $ with $ {\rm Re}(\tau)>0, $ is of the form 
\begin{equation*}
\label{rnmeinardus}
F(q)=\prod_{n=1}^{\infty}(1-q^n)^{-a_n}=\sum_{n=0}^{\infty}r(n)q^n,
\end{equation*} 
with $ a_n\ge 0. $ Under certain assumptions on which we do not elaborate now, Meinardus found asymptotic formulas for the coefficients $ r(n). $ More precisely, if the Dirichlet series $$ D(s)=\sum_{n=1}^{\infty}\frac{a_n}{n^s}\quad(s=\sigma
+it) $$ converges for $ \sigma>\alpha>0 $ and admits a meromorphic continuation to the region $ \sigma>-c_0~(0<c_0<1), $ region in which $ D(s) $ is holomorphic everywhere except for a simple pole at $ s=\alpha $ with residue $ A, $ then the following holds. 
\begin{Thm}[{Andrews \cite[Ch. 6]{And}, cf. Meinardus \cite{Mein}}]\setcounter{Thm}{1} 
\label{Meinardus}
As $ n\to\infty, $ we have
\[r(n)=cn^{\kappa}\exp\left( n^{\frac{\alpha}{\alpha+1}}\left(1+\frac{1}{\alpha}\right) ( A\Gamma(\alpha+1)\zeta(\alpha+1))^{\frac{1}{\alpha+1}}    \right)(1+O(n^{-\kappa_1}) ),  \]
where 
\begin{align*}
c & =  e^{D'(0)}\left(  2\pi(\alpha+1)\right) ^{-\frac12}\left(A\Gamma(\alpha+1)\zeta(\alpha+1) \right)^{\frac{1-2D(0)}{2+2\alpha}},\\
\kappa & =  \frac{2D(0)-2-\alpha}{2(\alpha+1)},\\
\kappa_1 & =  \frac{\alpha}{\alpha+1}\min \left\lbrace \frac{c_0}{\alpha}-\frac{\delta}{4},\frac12-\delta \right\rbrace,  
\end{align*}
with $ \delta>0 $ arbitrary.
\end{Thm}
Writing $ \tau=y-2\pi ix, $  an application of Cauchy's Theorem gives
\begin{equation}\label{SaddlePoint}r(n)=\frac1{2\pi i}\int_{\cal C}\frac{F(q)}{q^{n+1}}dq =e^{ny}\int_{-\frac12}^{\frac12}F(e^{-y+2\pi ix})e^{-2\pi inx} dx, \end{equation}
where $ \cal C $ is the (positively oriented) circle of radius $ e^{-y} $ around the origin.
Meinardus  found the estimate stated in Theorem \ref{Meinardus} by splitting the integral from \eqref{SaddlePoint} into two integrals evaluated over $ |x|\le y^{\beta} $ and over $ y^{\beta}\le|x|\le\frac12, $ for a certain choice of $ \beta $ in terms of $ \alpha, $ and by showing that the former integral gives the main contribution for the coefficients $ r(n), $ while the latter is only an error term. 
\par The positivity condition $ a_n\ge0 $ is, however, essential in Meinardus's proof and, as one can readily note, this is not satisfied by the factors from the product $ G_k(q) $. For this reason, we need to come up with a certain modification using the circle method and Wright's modular transformations \cite{WrightIII} for partitions into $k$th powers.  This will be used to show that the integral over $ y^\beta\le|x|\le \frac12 $ does not contribute. \par On comparing with what was done for the case $ k=2, $ the reader might notice that, up to this point, the strategy described here is analogous to that from \cite{AC}. The essential difference is that, in the case $ k=2, $ a numerical check (\cite[Lemma 5]{AC}) had to be carried out in order to prove a certain estimate (\cite[Lemma 6]{AC}). This numerical check was rather technical and certainly cannot be carried out for all $ k\ge2. $ In the present paper, we show how to avoid it by using a bound on Gauss sums due to Banks and Shparlinski \cite{BS} and by modifying a certain step in the argument from \cite{AC}. It is precisely this step that allows  a significantly simpler proof and, at the same time,  a generalization to any $ k\ge2. $

\subsection{Two estimates}\label{Estimates}

Keeping the notation introduced in the previous subsection and writing $ q=e^{-\tau}, $ with $ \tau=y-2\pi ix $ and  $ y>0, $  we recall that \begin{equation}\label{newG}
G_k(q)=\prod_{n=1}^{\infty}\frac{( 1-q^{2^kn^k}) ^2}{( 1-q^{2^{k+1}n^k}) ( 1-q^{n^k}) }.\end{equation}
Let $ s=\sigma+it $ and 
\[D_k(s)=\sum_{n=1}^{\infty}\frac{1}{n^{ks}}+\sum_{n=1}^{\infty}\frac{1}{(2^{k+1}n^k)^s}-2\sum_{n=1}^{\infty}\frac{1}{(2^kn^k)^s}=(1+2^{-s(k+1)}-2^{1-sk})\zeta(ks),\]
which is convergent for $ \sigma>\frac1k=\alpha, $ has a meromorphic continuation to $ \mathbb C $ (thus we may choose $ 0<c_0<1 $ arbitrarily) and a simple pole at $ s=\frac1k $ with residue $ A=\frac{1}k\cdot 2^{-\frac{k+1}k} .$ 
\par If $ \cal C $ is the (positively oriented) circle of radius $ e^{-y} $ around the origin, Cauchy's Theorem tells us that
\begin{equation}\label{CauchyCircle}
a_k(n)=\frac1{2\pi i}\int_{\cal C}\frac{G_k(q)}{q^{n+1}}dq =e^{ny}\int_{-\frac12}^{\frac12}G_k(e^{-y+2\pi ix})e^{-2\pi inx} dx, \end{equation} for $ n>0. $	
Set
\begin{equation}\label{defbeta}
\beta=1+\frac{\alpha}{2}\left( 1-\frac{\delta}{2}\right) ,\quad\text{with~} 0<\delta<\frac23,\end{equation}
so that 
\begin{equation}
\label{ineqbeta}
\frac{3k+1}{3k}<\beta<\frac{2k+1}{2k}, 
\end{equation}
and rewrite 
\begin{equation*}\label{epxressionfora2}
a_k(n)= I_k(n) +J_k(n),
\end{equation*}
where 
\[ I_k(n)=e^{ny}\int_{-y^{\beta}}^{y^{\beta}}G_k(q)e^{-2\pi inx}dx\qquad\text{and}\qquad
J_k(n)=e^{ny} \int_{y^{\beta}\le|x|\le\frac12}  G_k(q) e^{-2\pi inx}dx. \]
\par As already mentioned, the idea is that the main contribution for $ a_k(n) $ is given by $ I_k(n), $  and we will be able to prove this by using standard integration techniques. To show, however, that $ J_k(n) $ is an error term will prove to be much trickier. 
\section{The main term $ I_k(n) $}\label{MainTerm} In this section, we prove the following estimate.

\begin{Lem} 
\label{smallxmaintermG}
If $ |x|\le \frac12 $ and $ |\Arg(\tau)|\le\frac{\pi}{4}, $ 
then 
\[G_k( e^{-\tau}) =2^{-\frac{k-1}2} \exp\left( A\Gamma\left(\dfrac1k \right) \zeta\left(1+\dfrac1k \right) \tau^{-\frac1k}+O(y^{c_0})\right) \]
holds uniformly in $ x $ as $ y\to0, $ with $ 0<c_0<1. $ 
\end{Lem}
\begin{proof}
By taking logarithms in \eqref{newG}, we obtain \[\Log  (G_k( e^{-\tau}))  =\sum_{k=1}^{\infty}\frac 1k\sum_{n=1}^{\infty}\big( e^{-kn^k\tau}+e^{-2^{k+1}kn^k\tau}-2e^{-2^kkn^k\tau}\big) .\]
Using the Mellin inversion formula (see, e.g., \cite[p. 54]{Apostol2}) we get 
\[e^{-\tau}=\frac{1}{2\pi i}\int_{\sigma_0-i\infty}^{\sigma_0+i\infty}\tau^{-s}\Gamma(s)ds\] 
for $ {\rm Re}(\tau)>0$ and $\sigma_0>0, $  
thus
\begin{align}\label{integrand}
\Log  (G_k( e^{-\tau})) &=\frac{1}{2\pi i}\int_{\alpha+1-i\infty}^{\alpha+1+i\infty}\Gamma(s)\sum_{k=1}^{\infty}\frac 1k\sum_{n=1}^{\infty}( (kn^k\tau)^{-s}+(2^{k+1}kn^k\tau)^{-s}-2(2^kkn^k\tau)^{-s} )ds\nonumber\\ 
&=\frac{1}{2\pi i}\int_{\frac{k+1}k-i\infty}^{\frac{k+1}k+i\infty}\Gamma(s)D_k(s)\zeta(s+1)\tau^{-s}ds.
\end{align}
By assumption, \[|\tau^{-s}|=|\tau|^{-\sigma}e^{t\cdot\Arg( \tau)}\le|\tau|^{-\sigma}e^{\frac{\pi}{4}|t|}.\]
Well-known results (see, e.g., \cite[Ch. 1]{AAR} and \cite[Ch. 5]{Titch}) state that the bounds
\[
D_k(s) = O(|t|^{c_1}),\quad
\zeta(s+1) = O(|t|^{c_2}),\quad
\Gamma(s)= O\big( e^{-\frac{\pi|t|}{2}}|t|^{c_3} \big)
\]
hold uniformly in $ -c_0\le\sigma\le\frac{k+1}k =\alpha+1$  as $ |t|\to\infty, $ for some $ c_1,c_2, c_3>0, $
which means that we may shift the path of integration from $ \sigma=\alpha+1 $ to $ \sigma=-c_0. $ A quick computation gives 
$ D_k(0) =   0$ and $ D'_k(0)  =   -\frac{(k-1)\log2}2.$
The integrand in \eqref{integrand} has poles at $ s=\frac1k $ and $ s=0,$  with residues equal to 
$
\tau^{-\frac1k}$ and $
-\frac{(k-1)\log2}{2}
$ respectively, 
whereas the remaining integral equals 
\[\frac{1}{2\pi i}\int_{-c_0-i\infty}^{-c_0+i\infty}\tau^{-s}\Gamma(s)D(s)\zeta(s+1)ds\ll|\tau|^{c_0}\int_{0}^{\infty}t^{c_1+c_2+c_3}e^{-\frac{\pi t}{4}}dt\ll |\tau|^{c_0}=|y-2\pi ix|^{c_0}\le( \sqrt2y) ^{c_0},\]
since (again by the assumption) \[\frac{2\pi |x|}{y}=\tan (|\Arg(\tau)|)\le \tan\left(\frac{\pi}{4} \right) =1.\]
In conclusion, integration along the shifted contour gives
\[\Log  (G_k( e^{-\tau})) =\left( A\Gamma\left(\dfrac1k \right) \zeta\left(1+\frac1k \right) \tau^{-\frac1k}-\frac{(k-1)\log2}{2}\right) +O(y^{c_0}).\qedhere\] 
\end{proof}
\begin{xcom} This part is a straightforward generalization of \cite[Lemma 1]{AC}. We thought it best for the reader to keep the reasoning here as close as possible to that presented in \cite[\S 3.2]{AC}. On replacing $ k=2, $ the proof of \cite[Lemma 1]{AC} can be easily traced back.
\end{xcom}
\section{The error term $ J_k(n) $}\label{ErrorTerm} 
This section is dedicated to proving that $ J_k(n) $ does not contribute to the coefficients $ a_k(n). $ More precisely, we prove the following estimate. 
\begin{Lem}
\label{boundGbigoh} 
There exists $ \varepsilon>0 $ such that, as $ y\to0, $  
\begin{equation}\label{ErrorG}
G_k( e^{-\tau}) =O\left( \exp\left( A\Gamma\left(\dfrac1k \right) \zeta\left(1+\dfrac1k \right) y^{-\frac1k} -cy^{-\varepsilon} \right) \right) \end{equation}
holds uniformly in $ x $ with $ y^{\beta}\le |x|\le\frac 12, $ for some $ c>0. $ \end{Lem}
The proof is slightly more involved and will come in several steps. We start by describing the setup needed to apply the circle method. 
\subsection{Circle method} Inspired by Wright \cite{WrightIII}, we consider the Farey dissection of order $ \big\lfloor y^{-\frac k{k+1}} \big\rfloor $ of the circle $ \cal C $ over which we integrate in \eqref{CauchyCircle}. We further distinguish two kinds of arcs:
\begin{enumerate}[(i)]
	\item major arcs,  denoted $ \mathfrak M_{a,b}, $ such that $ b\le y^{-\frac1{k+1}}; $
	\item minor arcs, denoted $ \mathfrak m_{a,b}, $ such that $ y^{-\frac1{k+1}}<b\le y^{-\frac{k}{k+1}}. $
\end{enumerate}
\par In what follows, we express any $ \tau\in \mathfrak M_{a,b}\cup \mathfrak m_{a,b} $ in the form 
\begin{equation}\label{tautauprime}
\tau=y-2\pi ix=\tau'-2\pi i\frac ab,
\end{equation}
with $ \tau'=y-2\pi ix'. $ From basics of Farey theory it follows that 
\begin{equation}
\label{ineqxprim}
 |x'|\le\frac{y^{\frac k{k+1}}}{b}. 
\end{equation} 
For a neat introduction to Farey fractions and the circle method, the reader is referred to \cite[Ch. 5.4]{Apostol2}.
\subsection{Modular transformations} Recalling the definition of $ H_k(q), $ we can rewrite \eqref{newG} as
\begin{equation}
\label{NiceG}
G_k(q)=\frac{H_k(q)H_k(q^{2^{k+1}})}{H_k(q^{2^k})^2}.
\end{equation}
In order to obtain more information about $ G_k(q) ,$ we would next like to use Wright's transformation law \cite[Theorem 4]{WrightIII}  for the generating function $ H_k(q) $ of partitions into $k$th powers. \par Before doing so, we need to introduce a bit of notation. In what follows, $ 0\le a< b $ are assumed to be coprime positive integers, with $ b_1 $ the least positive integer such that $ b\mid b_1^2 $ and $ b=b_1b_2. $
First, set

\[j=j(k)=0,\qquad \omega_{a,b}=1\]  if $k$ is even, and
\[j=j(k)=\frac{(-1)^{\frac12(k+1)}}{(2\pi)^{k+1}}\Gamma(k+1)\zeta(k+1), \qquad \omega_{a,b}=\exp\left(\pi\left(\frac1{b^2}\sum_{h=1}^bhd_h-\frac14(b-b_2) \right)  \right) \] if $k$ is odd,
where  $ 0\le d_h<b $ is defined by the congruence \[ah^2\equiv d_h\pmod*{b}\] and
\[\mu_{h,s}=\begin{cases}
\frac{d_h}{b} &\text{if~}s\text{~is~odd},\\
\frac{b-d_h}{b}&\text{if~}s\text{~is~even},
\end{cases}\]	
for $ d_h\ne0. $ If $ d_h=0, $ we set $ \mu_{h,s}=1. $ Further, let 
\begin{equation}
\label{definitionsumS}
S_k(a,b)=\sum_{n=1}^b \exp\left( \frac{2\pi ia n^k}{b}\right)
\end{equation}be the so-called \textit{Gauss sum} (of order $k$), and
\begin{equation}
\label{definitionLambda}
\Lambda_{a,b}=\frac{\Gamma\left( 1+\frac1k\right)} {b}\sum_{m=1}^{\infty}\frac{S_k(ma,b)}{m^{1+\frac1k}}.
\end{equation}Finally, put
\[C_{a,b}={\left( \frac{b_1}{2\pi}\right)^{\frac k2}\omega_{a,b} ,}\]
and
\[P_{a,b}(\tau')=\prod_{h=1}^{b}\prod_{s=1}^k\prod_{\ell=0}^{\infty}\left( 1-g(h,\ell,s)\right)^{-1}, \]
with 
\[g(h,\ell,s)=\exp\left(\frac{(2\pi)^{\frac {k+1}k}(\ell+\mu_{h,s})^{\frac 1k}e^{\frac{\pi i}{2k}(2s+k+1)}}{b\sqrt[k]{\tau'}}-\frac{2\pi ih}{b} \right).  \]
Having introduced all the required objects, we can now state Wright's modular transformation \cite[Theorem 4]{WrightIII}, which says, in our notation, that
\begin{equation}\label{Wrightexp}
H_k(q)=H_k\left( e^{\frac{2\pi ia}{b}-\tau'} \right)=C_{a,b}\sqrt{\tau'}e^{j\tau'}\exp\left( {\frac{\Lambda_{a,b}}{\sqrt[k]{\tau'}}}\right)P_{a,b}(\tau').  \end{equation}
On combining \eqref{NiceG} and \eqref{Wrightexp} we obtain, for some positive constant $ C $ that can be made explicit if necessary, 
\begin{equation}\label{formforG}
G_k(q)=Ce^{j\tau'}\exp\left( \frac{\lambda_{a,b}}{\sqrt[k]{\tau'}}\right) \frac{P_{a,b}(\tau')P_{a,b}'(2^{k+1}\tau')}{P_{a,b}''(2^k\tau')^2},
\end{equation}
where
\[P'_{a,b}=P_{\frac{2^{k+1}a}{(b,2^{k+1})},\frac{b}{(b,2^{k+1})}},\quad P_{a,b}''=P_{\frac{2^ka}{(b,2^k)},\frac{b}{(b,2^k)}}\]
and 
\begin{equation}
\label{lambda_ab}
\lambda_{a,b}=\Lambda_{a,b}+2^{-\frac{k+1}k} \Lambda_{\frac{2^{k+1}a}{(2^{k+1},b)},\frac{b}{(2^{k+1},b)}}-\Lambda_{\frac{2^ka}{(2^k,b)},\frac{b}{(2^k,b)}}.
\end{equation}
\subsection{Gauss sums} As we shall soon see, a crucial step in our proof is finding an upper bound for $ {\rm Re}( \lambda_{a,b}) $ or, what is equivalent, a bound for $ |\lambda_{a,b}|. $ This is given by the following sharp estimate found by Banks and Shparlinski \cite{BS} for the Gauss sums defined in \eqref{definitionsumS}. 
\begin{Thm}[{\cite[Theorem 1]{BS}}]\setcounter{Thm}{1}  For any coprime positive integers $ a,b $ with $ b\ge2 $ and any $ k\ge2, $ we have 
\begin{equation}
\label{Shparli}
|S_k(a,b)|\le \cal A b^{1-\frac1k},
\end{equation}
where $ \cal A=4.709236\dots. $ 
\end{Thm}
The constant $ \cal A $ is known as \textit{Stechkin's constant}. Stechkin \cite{Stech} conjectured in 1975 that the quantity $$ \cal A=\sup\limits_{b,n\ge2}\max\limits_{(a,b)=1}\frac{|S_k(a,b)|}{b^{1-\frac1k}} $$ is finite, this being proven in 1991 by Shparlinski \cite{Shparlinski}. In the absence of any effective bounds on the sums $ S_k(a,b), $ the precise value of $ \cal A $ remained a mystery until 2015 when, using the work of Cochrane and Pinner \cite{CP} on Gauss sums with prime moduli and that of Cohn and Elkies \cite{CE} on lower bounds for the center density in the sphere packing problem, Banks and Shparlinski \cite{BS} were finally able to determine it.  
Coming back to our problem, we can now prove the following estimate.
\begin{Lem}
\label{boundmaxReIm}
If $0\le a<b$ are coprime integers with $ b\ge2, $ we have \[| \lambda_{a,b}| < 3  \cal A \cdot  \Gamma\left( 1+\frac 1k\right)\zeta\left(1+\frac1k \right)b^{-\frac1k}\sum_{d|b}\frac1d  ,  \]
where $ \cal A $ is Stechkin's constant. 
\end{Lem}
\begin{proof}
Let us first give a bound for $ |\Lambda_{a,b}|. $ If we recall \eqref{definitionLambda} and write $\Lambda_{a,b}=\Gamma\left(1+\frac 1k \right)  \Lambda^*_{a,b} , $
we have, on using the fact that $ S_k(ma,b)=d S_k\left(\frac{ma}d,\frac bd\right) $ to prove the second equality below, and on replacing $ m\mapsto md $ and $ d\mapsto \frac bd $ to prove the third and fourth respectively,   
\begin{align*}
\Lambda^*_{a,b}&=\frac 1b\sum_{m=1}^{\infty}\frac{S_k(ma,b)}{m^{1+\frac1k}}=\frac 1b\sum_{d\mid b}\sum_{\substack{m\ge1\\(m,b)=d}}\frac{d S_k\left(\frac{ma}d,\frac bd\right)}{m^{1+\frac1k}}=\frac 1b\sum_{d|b}d\sum_{\substack{m\ge1\\(m,b/d)=1}}\frac{S_k(m a, \frac bd)}{(m d)^{1+\frac1k}}\\
&=\frac 1b\sum_{d\mid b}d^{-\frac1k}\sum_{\substack{m\ge1\\(m,b/d)=1}}\frac{S_k(m a, \frac bd)}{m^{1+\frac1k}}=\frac 1b\sum_{d|b}\left( \frac bd\right)^{-\frac1k} \sum_{\substack{m\ge1\\(m,d)=1}}\frac{S_k(ma,d)}{m^{1+\frac1k}}=\frac{1}{b^{1+\frac1k}}\sum_{d|b}d^{\frac1k}\sum_{\substack{m\ge1\\(m,d)=1}}\frac{S_k(ma,d)}{m^{1+\frac1k}}.
\end{align*}
On invoking \eqref{Shparli}, we obtain
\begin{align*}
|\Lambda_{a,b}|&\le\frac{\Gamma\left( 1+\frac1k\right) }{b^{1+\frac1k}}\sum_{d|b}d^{\frac1k}\sum_{\substack{m\ge1\\(m,d)=1}}\frac{|S_k(ma,d)|}{m^{1+\frac1k}}\le 
\frac{\cal A\Gamma\left(1+\frac1k \right)\zeta\left(1+\frac1k \right)  }{b^{\frac1k}}\sum_{d|b}\frac1{d^k},
\end{align*}
from where the claim follows easily by applying this bound to the expression for $ \lambda_{a,b}  $  from \eqref{lambda_ab}.
\end{proof}
\begin{rem}
The fact that $S_k(a,b)\ll b^{1-\frac1k}$ is known, for instance, from the book of Vaughan; see Theorem 4.2 in \cite[p. 47]{Vaughan}. This means that instead of the Stechkin constant $ \cal A $ in Lemma \ref{boundmaxReIm}, we would have a constant $ \cal A_k $ depending on $ k, $ which would already be enough for our purposes, as it will be revealed shortly in the proof of Lemma \ref{boundGbigoh}. While a universal bound like Stechkin's constant $ \cal A $ is not needed for the proof and a constant depending on $ k $ would suffice, we thought it instructive to give this neater argument and to point out an interesting connection between power  partitions and the sphere packing problem.
\end{rem}
\subsection{Final estimates} We are now getting closer to our purpose and we only need a few last steps before giving the proof of Lemma \ref{boundGbigoh}. Let us begin by estimating the factors of the form $ P_{a,b} $ appearing in \eqref{formforG}.
\begin{Lem}
\label{majorarcs} If $ \tau\in\mathfrak M_{a,b}\cup\mathfrak m_{a,b}, $ then
 
\[\log|P_{a,b}(\tau')|\ll b\quad\text{as $ y\to 0. $}\]
\end{Lem}
\begin{proof}
Using \eqref{ineqxprim} and letting $ y\to0 $, we have \[|\tau'|^{1+\frac1k}=(y^2+4\pi^2x'^2)^{\frac{k+1}{2k}}\le\left(y^2+\frac{4\pi^2y^{\frac{2k}{k+1}}}{b^2} \right)^{\frac{k+1}{2k}}\le\frac{c_4y}{b^{\frac{k+1}k}}=\frac{c_4 {\rm{Re}}\left(\tau'\right) }{b^{\frac{k+1}k}}, \]
for some $ c_4>0. $ Thus, \cite[Lemma 4]{WrightIII} gives
\[|g(h,\ell,s)|\le e^{-c_5(\ell+1)^{\frac1k}},\]
with $ c_5=\frac{4\sqrt[k]{2\pi}}{kc_4}, $ which in turn leads to
\[| \log| P_{a,b}(\tau') || \le \sum_{h=1}^{b}\sum_{s=1}^{k}\sum_{\ell=1}^{\infty}|\log(1-g(h,\ell,s))|\le kb\sum_{\ell=1}^{\infty}\big| \log\big( 1-e^{-c_5(\ell+1)^{\frac1k}}\big)\big| \ll b,\] concluding the proof.
\end{proof}
The next result gives a bound for $ G_k(q) $ on the minor arcs. As it is an immediate consequence of replacing $a=\frac1k, $ $ b=\frac1{k+1}, $ $ c=2^{k-1}, $ $ \gamma=\varepsilon $ and $ N=y^{-1} $ in \cite[Lemma 17]{WrightIII}, we omit its proof.
\begin{Lem}
\label{minors}
If $ \varepsilon>0 $ and $ \tau\in\mathfrak m_{a,b}, $ 
then
\begin{equation*}\label{correctedminorarcs}
|\Log (G(q))|\ll_{\varepsilon}{y^{\frac{k2^{k-1}-k-1}{k(k+1)}-\varepsilon}}.
\end{equation*}  
\end{Lem}
\begin{rem}
Note that $ k^{2^{k-1}}>k+1 $ for any $ k\ge2, $ therefore the exponent of $ y $ in Lemma \ref{minors} is positive for a small enough choice of $ \varepsilon>0 $.
\end{rem}
At last, we need the following estimate, a modified version of \cite[Lemma 6]{AC}.

\begin{Lem}
\label{biglemma}
If $0\le a<b$ are coprime integers with $ b\ge2 $ and $ x\notin\bb Q, $ we have as $ y\to 0, $  for some $ c>0,$  
\begin{equation}\label{boundRe}
{\rm{Re}}\left( \frac{\lambda_{a,b}}{\sqrt[k]{\tau'}} \right)\le \frac{\lambda_{0,1}-c}{\sqrt[k] y}.  
\end{equation}

\end{Lem}
\begin{proof}Note that \[\lambda_{0,1}=2^{-\frac{k+1}k}\Lambda_{0,1} =\frac1{2^{1+\frac1k}}\Gamma\left( 1+\frac 1k\right)\zeta\left(1+\frac1k \right)=A\Gamma\left(\dfrac1k \right) \zeta\left(1+\dfrac1k \right).\]
Writing $ \tau'=y+ity  $ for some $ t\in\bb R, $ we have
\begin{align*}
{\rm{Re}}\left( \frac{\lambda_{a,b}}{\sqrt[k]{\tau'}} \right)&=\frac{1}{\sqrt[k]y}{\rm Re}\left(\frac{\lambda_{a,b}}{\sqrt[k]{1+it }} \right)=\frac{1}{\sqrt[k]y}{\rm Re}\left( \frac{\lambda_{a,b}}{\sqrt[2k]{1+t^2}e^{\frac ik\arctan t}}\right)\\
&= \frac{1}{\sqrt[k]y\sqrt[2k]{1+t^2}}\left( \cos\left(\frac{\arctan t}{k}\right) {\rm Re}\left( \lambda_{a,b}\right)+\sin  \left(\frac{\arctan t}{k}\right){\rm Im} \left( \lambda_{a,b}\right)\right).
\end{align*}
If we denote  \[f_k(t)=\frac{1}{\sqrt[2k]{1+t^2}}\left( \cos\left(\frac{\arctan t}{k}\right) {\rm Re}\left( \lambda_{a,b}\right)+\sin  \left(\frac{\arctan t}{k}\right){\rm Im} \left( \lambda_{a,b}\right)\right),\]
we clearly have $ f_k(t)\to0 $ as $ |t|\to\infty. $
Note now that the choice of $ x $ is independent from that of $ y, $ and recall from \eqref{tautauprime} that $ \tau'=y-2\pi ix', $ with $ x'=x-\frac ab, $ hence $ t=-\frac{x'}{2\pi y}. $ The assumption $ x\notin \bb Q $ implies $ x'\ne0, $ and so $ |t|\to\infty $ as $ y\to0. $ Consequently, we have $ f_k(t)\to 0 $ as $ y\to0. $ In combination with Lemma \ref{boundmaxReIm} and the well-known fact that $ \sigma_0(n)=o(n^{\epsilon}) $ for any $ \epsilon>0 $ (for a proof see, e.g., \cite[p. 296]{Apostol}), where $ \sigma_0(n) $ denotes the number of divisors of $ n, $ this completes the proof. 
 \end{proof}


\subsection{Estimate for $J_k(n)$} We are now equipped with all the machinery needed for Lemma \ref{boundGbigoh}. 
\begin{proof}[Proof of Lemma \ref{boundGbigoh}]  If $\tau\in\mathfrak m_{a,b},$ then it suffices to apply Lemma \ref{minors} (because, as $ y\to 0, $ a negative power of $ y $ will dominate any positive power of $ y; $ in particular, also the term $ jy $ coming from the factor $ e^{j\tau'} $ in the case when $k$ is odd), so let us assume that $ \tau\in\mathfrak M_{a,b}. $ 
\par We first consider the behavior near 0, corresponding to $ a=0, $ $ b=1,$ $ \tau=\tau'=y-2\pi ix. $ 
Writing $ y^{\beta}=y^{\frac{2k+1}{2k}-\varepsilon} $ with $ \varepsilon>0 $ (here we use the second inequality from \eqref{ineqbeta}), we have, on setting $ b=1 $ in \eqref{ineqxprim}, 
\begin{equation}
\label{cusp0}
y^{\frac{2k+1}{2k}-\varepsilon}\le |x|=|x'|\le y^{\frac k{k+1}}.
\end{equation}
By \eqref{formforG} we get
\[G_k(q)=Ce^{j\tau}\exp \left( \frac{\lambda_{0,1}}{\sqrt[k]\tau}\right)  \frac{P_{0,1}(\tau)P_{0,1}(2^{k+1}\tau)}{P_{0,1}(2^k\tau )^2}  \]
for some $ C>0 $. Thus, by Lemma \ref{majorarcs} we obtain
\[\log|G_k(q)|=\frac{\lambda_{0,1}}{\sqrt[k]{|\tau|}}+jy+O(1).\]
Using \eqref{cusp0} to prove the first inequality below and expanding into Taylor series to prove the second, we have, on letting $ y\to0, $
\[\frac{1}{\sqrt[k]{|\tau|}}=\frac{1}{\sqrt[k]{y}}\frac{1}{\left( 1+\frac{4\pi^2x^2}{y^2}\right)^{\frac1{2k}} }\le \frac{1}{\sqrt[k]{y}}\frac{1}{\left( 1+4\pi^2y^{\frac1k-2\varepsilon}\right) ^{\frac1{2k}}}\le \frac{1}{\sqrt[k]{y}}\left( 1-c_6y^{\frac1k-2\varepsilon}\right) \] for some $ c_6>0, $
and this concludes the proof in this case. 

To complete the proof, let $ \tau\in\mathfrak M_{a,b}, $ with $ 2\le b\le y^{-\frac1{k+1}}. $ We distinguish two cases. First, let us deal with the case when $ x\notin\bb Q. $ By \eqref{formforG} and Lemma \ref{majorarcs} we obtain 
\begin{equation}\label{FullExpress}
\log|G_k(q)|={\rm Re} \left( \frac{\lambda_{a,b}}{\sqrt[k]{\tau'}} \right)+jy+O\left( y^{-\frac1{k+1}}\right) ={\rm Re} \left( \frac{\lambda_{a,b}}{\sqrt[k]{\tau'}} \right)+O\left( y^{-\frac1{k+1}}\right)    \end{equation}
as $ y\to0. $ Since by Lemma \ref{biglemma} there exists $ c_7>0 $ such that 
\begin{equation}\label{ND}
{\rm Re} \left( \frac{\lambda_{a,b}}{\sqrt[k]{\tau'}} \right)\le \frac{\lambda_{0,1}-c_7}{\sqrt[k] y},
\end{equation}
we infer from \eqref{ND} that, as $ y\to 0, $  we have
\[\log|G_k(q)|\le \frac{\lambda_{0,1}-c_8}{\sqrt[k] y}  \]
for some $ c_8>0 $
and the proof is concluded under the assumption that $ x\notin\bb Q $.
\par Finally, assume that $ x=\frac ab, $ that is, $ x'=0 $ and $ \tau=y-2\pi i\frac ab. $  We claim that the estimate \eqref{ErrorG} is satisfied with the same implied constant, call it $ C_1. $ Suppose by sake of contradiction that this is not the case. Then there exist infinitely small values of $ y>0 $ for which 
\[|G_k( e^{-\tau})| \ge C_2 \exp\left( \frac{\lambda_{0,1}}{\sqrt[k]y}   -cy^{-\varepsilon} \right), \] with $ C_2>C_1. $
However, we can pick now $ x'\notin\bb Q $ infinitely small and set $ \tau_1=y-2\pi i\left(x'+\frac ab \right).  $ For a fixed choice of $ y, $ we have $ t\to 0 $ as $ x'\to 0; $ thus, by the same calculation done in the proof of Lemma \ref{biglemma}, we obtain 
\begin{equation}\label{LimitRe}
{\rm{Re}}\left( \frac{\lambda_{a,b}}{\sqrt[k]{\tau'_1}} \right)\to {\rm{Re}}\left( \frac{\lambda_{a,b}}{\sqrt[k]{y}} \right)={\rm{Re}}\left( \frac{\lambda_{a,b}}{\sqrt[k]{\tau'}} \right), \end{equation}
since $ f_k(t)\to 1. $  On noting that $ {\rm Re}(\tau_1')={\rm Re}(\tau)=y, $ while clearly all factors of the form  $ |P_{a,b}(k\tau_1')|$ tend to $|P_{a,b}(k\tau')| $ as $ x'\to0, $  we obtain a contradiction, in the sense that, on one hand, \eqref{FullExpress} and \eqref{LimitRe} yield
\[|G_k(e^{-\tau_1})|\to|G_k(e^{-\tau})|\] 
as $ x'\to0, $ 
whereas on the other, for a sufficiently small choice of $ y>0, $ we have 
\[|G_k(e^{-\tau})|-|G_k(e^{-\tau_1})|\ge(C_2-C_1) \exp\left( \frac{\lambda_{0,1}}{\sqrt[k]y}   -cy^{-\varepsilon} \right),\]
quantity which gets arbitrarily large for sufficiently small choices of $ y>0. $  
\end{proof}
\begin{xcom}
It is in this part where our proof differs substantially from that given in \cite{AC} in the case $ k=2. $ More precisely, \cite[Lemma 5]{AC} was needed to prove the inequality \eqref{boundRe} for all values of $ y, $ inequality which was then used in the estimates made in the proof of \cite[Lemma 2]{AC}, the equivalent of Lemma \ref{boundGbigoh} from the present paper. However, we are only interested in establishing the estimates from Lemma \ref{boundGbigoh} on letting $ y\to0, $ which is why we only need the bound  \eqref{boundRe} to hold as $ y\to0. $ The argument presented in Lemma \ref{biglemma} further tells us that, in order for this to happen, the estimate \eqref{Shparli}, obtained using the bound on Gauss sums found by Banks and Shparlinski \cite{BS}, is enough. As a consequence, we can avoid the rather involved numerical check done in \cite[Lemma 5]{AC}, a check which we would, in fact, not even be able to implement for all values $ r\ge2. $ In particular, the present argument gives a simplified proof of the results from \cite{AC}.  
\end{xcom}
\section{Analytic proof of Theorem \ref{ThmPowers}}\label{Proof of Theorem 2}
In this section we give the analytic proof of Theorem \ref{ThmPowers}.
Having already proven the two estimates from Lemmas \ref{smallxmaintermG}--\ref{boundGbigoh}, the rest is only a matter of careful computations. The reader is reminded that, because of the reformulation from Section \ref{reformulation}, what we are interested in is computing asymptotics for the coefficients
\begin{equation}
\label{Integral}
a_k(n)=e^{ny}\int_{-\frac12}^{\frac12}G_k(e^{-y+2\pi ix})e^{-2\pi inx} dx.
\end{equation}
\subsection{Saddle-point method}  Recall that, as defined in Section \ref{Estimates}, we denote $ \alpha=\frac 1k $ and $ A=\frac1k\cdot 2^{-\frac{k+1}k} ,$  notation which we keep, for simplicity, in what follows. Before delving into the proof, we make a particular choice for $ y $ as a function of $ n. $ More precisely, let 
\begin{equation}\label{choicey}
y=n^{-\frac 1{\alpha+1}}(  A\Gamma(\alpha+1) \zeta(\alpha+1)  )^{\frac 1{\alpha+1}}=n^{-\frac k{k+1}}\left(  A\Gamma\left(\dfrac1k \right) \zeta\left(1+\dfrac1k \right)  \right)^{\frac k{k+1}}, \end{equation}
and write
$m=ny.$ 
\par
The reason for this choice of $ y $ is motivated by the \textit{saddle-point} method and becomes clear once the reader recognizes in \eqref{choicey} the quantity appearing in Lemmas \ref{smallxmaintermG} and \ref{boundGbigoh}. As the maximum absolute value of the integrand from \eqref{Integral} occurs for $ x=0, $ around which point Lemma \ref{smallxmaintermG} tells us that the integrand is well approximated by \[\exp(  A\Gamma(\alpha) \zeta(\alpha+1)y^{-\alpha}+ny), \] the saddle-point method suggests maximizing this expression, that is, finding the value of $ y $ for which \[\frac d{dy}(\exp(  A\Gamma(\alpha) \zeta(\alpha+1)y^{-\alpha}+ny))=0.\] 
\subsection{Completing the proof} We have now all ingredients necessary to conclude the proof of Theorem \ref{ThmPowers}. The proof merely consists of a skillful computation, which can be carried out in two ways. Since Lemma \ref{smallxmaintermG} and Lemma \ref{boundGbigoh} are completely analogous to the two estimates found by Meinardus (combined in the \textit{Hilfssatz} from \cite[p. 390]{Mein}), one way is to follow his approach and carry out the same computations done in \cite[pp. 392--394]{Mein}. The second way is slightly more explicit and is based entirely on the computation done in the proof of the case $ k=2 $ from \cite[pp. 139--141]{AC}. For sake of completeness and for comparison with the corresponding computation, we will sketch in what follows the main steps of the argument, while leaving some details and technicalities as an exercise for the interested reader.    
%
\begin{proof}[Analytic proof of Theorem \ref{ThmPowers}] We begin by proving  the inequalities \eqref{IneqsPowers}.
By Lemma \ref{boundGbigoh} and \eqref{choicey} we have
\begin{align*}\label{correctbounderror}
J_k(n)&=e^{ny} \int_{y^{\beta}\le|x|\le\frac12}  G( e^{-y+2\pi ix}) e^{-2\pi inx}dx= e^{ny} \int_{y^{\beta}\le|x|\le\frac12} O \Big( e^{y^{-\alpha}A\Gamma( \alpha) \zeta(\alpha+1 )-cy^{-\varepsilon}}\Big)  dx\nonumber\\&= 
e ^{ny} \cdot O\left( e^{y ^{-\alpha}A\Gamma( \alpha) \zeta(\alpha+1 )-cy ^{-\varepsilon}}\right)  =
O \Big( e^{n^{\frac{\alpha}{\alpha+1}}\left( 1+\frac1\alpha\right)(A\Gamma(\alpha+1)\zeta(\alpha+1))^{\frac1{\alpha+1}} -C_1n^{\varepsilon_1}}\Big)  
\end{align*}
as $ n\to0, $ with $ \varepsilon_1=\frac{k\varepsilon}{k+1}>0 $ and some $ C_1>0. $
\par We now compute the main asymptotic contribution, which will be given by $ I_k(n) $. Let $ n\ge  n_1 $ be large enough so that $y^{\beta-1}\le\frac{1}{2\pi}. $ This choice allows us to apply Lemma \ref{smallxmaintermG}, as it ensures $ |x|\le\frac12 $ and $ |\Arg (\tau)|\le \frac{\pi}{4}. $  From Lemma \ref{smallxmaintermG} we obtain
\begin{equation}\label{correctmainterm}
I_k(n)= \frac{e^{ny}}{2^{\frac{k-1}2}} \int_{-y^{\beta}}^{y^{\beta}}  e^{A\Gamma(\alpha) \zeta(\alpha+1) \tau^{-\alpha} +O(y^{\varepsilon})-2\pi inx }dx. \end{equation}
Writing 
\[\tau^{-\alpha}=\frac{1}{\sqrt[k]{\tau}}=\frac{1}{\sqrt[k]{y}}+\left( \frac{1}{\sqrt[k]{\tau}}-\frac{1}{\sqrt[k]{y}}\right), \]
we can further express \eqref{correctmainterm} as 
\begin{align*}
I_k(n)&= \frac{e^{ny}}{2^{\frac{k-1}2}}\int_{-y^{\beta}}^{y^{\beta}} e^{A\Gamma(\alpha) \zeta(\alpha+1)\frac{1}{\sqrt[k]y}} e^{A\Gamma(\alpha) \zeta(\alpha+1)\left( \frac{1}{\sqrt[k]{\tau}}-\frac1{\sqrt[k]y}\right)  }e^{-2\pi inx+O(y^{c_0})}dx\\
&=\frac{e^{\left(1+\frac1\alpha \right) n^{\frac\alpha{\alpha+1}}(A\Gamma(\alpha+1) \zeta(\alpha+1))^{\frac1{\alpha+1}} }}{2^{\frac{k-1}2}}\int_{-y^{\beta}}^{y^{\beta}}e^{\frac{A\Gamma(\alpha) \zeta(\alpha+1)}{\sqrt[k] y}\Big( \frac{1}{\sqrt[k]{1-\frac{2\pi ix}{y}}}-1\Big)  }e^{-2\pi inx+O(y^{c_0})}dx.
\end{align*} 
With $ u=-\frac{2\pi x}{y}, $ we obtain 
\begin{equation}\label{integralfinal}
I_k(n)=\frac{ye^{\left(1+\frac1\alpha \right) n^{\frac\alpha{\alpha+1}}(A\Gamma(\alpha+1) \zeta(\alpha+1))^{\frac1{\alpha+1}} }}{2\pi\cdot 2^{\frac{k-1}2}}\int_{-2\pi y^{\beta-1}}^{2\pi y^{\beta-1}}e^{\frac{A\Gamma(\alpha) \zeta(\alpha+1)}{\sqrt[k] y}\left( \frac{1}{\sqrt[k]{1+iu}}-1\right)  +inuy+O(y^{c_0})}dx.
\end{equation}
Set, for simplicity, $ B=A\Gamma(\alpha) \zeta(\alpha+1). $ 
We have the Taylor series expansion
\[
\frac{1}{\sqrt[k]{1+iu}}=1-\frac{iu}{k}-\frac{(k+1)u^2}{2k^2}+O(|u|^3), \]
from where, on recalling that $ |u|\le2\pi y^{\beta-1} $ and using \eqref{choicey} to compute $ B=kny^{1+\frac1k}, $ it follows that   \begin{align*}B\frac{1}{\sqrt[k] y}\left( \frac{1}{\sqrt[k]{1+iu}}-1\right)+inuy&=-\frac{Biu}{k\sqrt[k] y}+inuy-\frac{(k+1)Bu^2}{2k^2\sqrt[k]y}+O\left(\frac{|u|^3}{\sqrt[k] y} \right)\\&=\frac{(k+1)Bu^2}{2k^2\sqrt[k]y}+O\left(n^{\frac1{k+1}\left(1+\frac{3(1-\beta)}{\alpha} \right) }\right).\end{align*}
For an appropriate constant $ C_2, $ we may then change the integral from the right-hand side of \eqref{integralfinal} into
\begin{align*}
\int_{|u|\le 2\pi y^{\beta-1}}e^{B\frac{1}{\sqrt[k] y}\left( \frac{1}{\sqrt[k]{1+iu}}-1\right) +inuy +O(y^{c_0} ) }du
&=\int_{|u|\le C_2 }e^{-\frac{(k+1)Bu^2}{2k^2\sqrt[k] y}}  e^{O\left( y^{c_0}+\frac{|u|^3}{\sqrt[k] y}\right) }   du\\
&=\int_{|u|\le C_2  }e^{-\frac{(k+1)Bu^2}{2k^2\sqrt[k] y}}   e^{O\left( n^{-\frac{kc_0}{k+1}}+ n^{\frac{1+3k(1-\beta)}{k+1}}\right) }   du\\
&=\int_{|u|\le C_2 }e^{-\frac{(k+1)Bu^2}{2k^2\sqrt[k] y}}\bigg( 1+\bigg( e^{O\left( n^{-\frac{kc_0}{k+1}}+ n^{\frac{1+3k(1-\beta)}{k+1}}\right) }-1\bigg) \bigg)    du.
\end{align*}
From the first inequality in \eqref{ineqbeta}, we see that $ 1+3k(1-\beta)<0, $ and thus  
\[e^{O\left( n^{-\frac{kc_0}{k+1}}+ n^{\frac{1+3k(1-\beta)}{k+1}}\right) }-1=e^{O\big( n^{-\frac{kc_0}{k+1}}+n^{-\frac16+\frac{\delta}4}\big) }-1=O(n^{-\kappa} ),  \]
where $ \kappa=\frac1{k+1}\min\left\lbrace kc_0,\frac12-\frac{3\delta}4\right\rbrace . $ We further get, on using \eqref{defbeta} when changing the limits of integration,
\begin{align}
\int_{|u|\le 2\pi y^{\beta-1}}e^{B\frac{1}{\sqrt y}\left( \frac{1}{\sqrt{1+iu}}-1\right) +inuy +O(y^{c_0} ) }du
&=\int_{|u|\le C_2  }e^{-\frac{(k+1)Bu^2}{2k^2\sqrt[k] y}}(  1+ O( n^{-\kappa}) )  du\nonumber\\
&=c(n)\int_{|v|\le C_3\cdot n^{\frac{\delta}{4(k+1)}}} e^{-v^2}(  1+ O( n^{-\kappa}) )dv,\label{Gauss}
\end{align}
where $c(n)=\sqrt{\frac{2k}{k+1}}(\alpha Bn^{\alpha})^{-\frac1{2(\alpha+1)}}$
and $ C_3>0 $ is a constant.
By letting $ n\to\infty, $ and turning the integral from \eqref{Gauss} into a Gauss~integral, we obtain \begin{equation}\label{GaussError}
\int_{|u|\le 2\pi y^{\beta-1}}e^{B\frac{1}{\sqrt[k] y}\left( \frac{1}{\sqrt[k]{1+iu}}-1\right) +inuy +O(y^{c_0} ) }du=c(n) \sqrt{\pi}(  1+ O( n^{-\kappa_1}) ),\end{equation}
where $ \kappa_1=\frac1{k+1}\min\left\lbrace kc_0-\frac{\delta}{4},\frac12-\delta \right\rbrace . $ Putting together \eqref{integralfinal}, \eqref{Gauss} and \eqref{GaussError} we see that, as predicted by Meinardus (Theorem 3),
the main asymptotic contribution for our coefficients is given by
\begin{equation}\label{p2diff}
a_k(n)\sim Cn^{-\frac{\alpha+2}{2(\alpha+1)}}e^{n^{\frac{\alpha}{\alpha+1}}\left( 1+\frac1\alpha\right)(A\Gamma(\alpha+1)\zeta(\alpha+1))^{\frac1{\alpha+1}} } ,
\end{equation}
where \[C=\frac{1}{\sqrt{2^k(\alpha+1)\pi}}(A\Gamma(\alpha+1)\zeta(\alpha+1))^{\frac1{2(\alpha+1)}}. \]
This shows that the inequalities in \eqref{IneqsPowers} are true for $ n\to\infty. $ However, the combinatorial proof given in Section \ref{combiproof} shows that the inequalities hold for every $ n\ge1 $ (apart perhaps, for each $ k\ge2, $ from a few small values of $ n $ for which the inequalities might turn to equalities).  The proof can be now completed in two ways: either by adding the estimate \eqref{p2diff} for $ a_k(n)=(-1)^n(p_k(0,2,n)-p_k(1,2,n)) $ to that obtained by Wright  for $ p_k(n)=p_k(0,2,n)+p_k(1,2,n) $ (see \cite[Theorem 2]{WrightIII}), or by invoking the recent work of Zhou (see \cite[Corollary 1.2]{Zhou}).\end{proof}

\section{Open questions}\label{Open}
It would be of interest to see if Theorems \ref{Conj1} and \ref{ThmPowers} admit analogues for moduli $ m\ge3 $. Another conjecture formulated in the unpublished manuscript of Bringmann and Mahlburg \cite{BM} states the following. 
\begin{Con}[Bringmann--Mahlburg, 2012]
As $ n\to\infty $, we have 
\begin{equation}\label{Mod3}
p_2(0,3,n)\sim p_2(1,3,n)\sim p_2(2,3,n)\sim\frac{p_2(n)}{3}\end{equation}
and
\begin{equation}\label{Mod3Ineqs}
\begin{cases}
p_2(0,3,n)>p_2(1,3,n)>p_2(2,3,n) & \text{if~} n\equiv0\pmod*3, \\
p_2(1,3,n)>p_2(2,3,n)>p_2(0,3,n) & \text{if~} n\equiv1\pmod*3, \\
p_2(2,3,n)>p_2(0,3,n)>p_2(1,3,n) & \text{if~} n\equiv2\pmod*3.
\end{cases}
\end{equation}
\end{Con}
Indeed, the work of Zhou \cite{Zhou} proves the  equidistribution statement from \eqref{Mod3}. However, while the inequalities \eqref{Mod3Ineqs} hold for small of values of $ n, $ numerical experiments seem to reveal the fact that the pattern loses its structure as $ n $ grows larger, and that the signs of the inequalities change. In this regard, the following variant of this question seems more reasonable. \par Let $ S\subseteq \bb N. $
If for any $ n $ we arrange the $p_S(a,m,n)$ in non-increasing order (and, in case there are $0\le i<j \le m-1$ such that  $p_S(i,m,n)=p(j,m,n),$ we write $ p_S(i,m,n) $ before $ p_S(j,m,n) $ and we keep this choice consistent throughout), we obtain an $ m $-tuple $ (p_S(i_0,m,n),\ldots,p_S(i_{m-1},m,n)), $ with $ \{i_0,i_1,\ldots,i_{m-1}\}=\{0,1,\ldots, m-1\}. $ Let we define a sequence $ \{u_n\} $ of such $ m $-tuples by setting $ u_n=(i_0,i_1,\ldots,i_{m-1}), $ where the order is taken into account, and with the arrangement choice assumed above.

\begin{Quest}\label{Q1}
For the set $ S=\{n^k:n\in\bb N\} $ of perfect $ k $th powers or, more generally, for sets $ S=\{f(n)\}_{n\ge1} $ of polynomial functions as those in \cite{Liardet,RothSzek}, do the sequence $\{u_n\}_{n\ge1}$ thus constructed become periodic?
\end{Quest} 
\begin{Quest}
If so, is the statement of Question \ref{Q1} true for  \textit{all} $ m\ge3 ?$ 
\end{Quest}
Finally, we note that, although we could not directly apply Meinardus's Theorem to our problem, we did end up nevertheless with the two similar estimates, obtaining the asymptotics that his theorem would have heuristically predicted. This naturally leads to the following question.
\begin{Quest}
Can Meinardus's Theorem be strengthened  so as to deal with a more general class of infinite product generating functions than that studied in \cite{Mein}?
\end{Quest}

\section*{Acknowledgments} The author would like to thank Kathrin Bringmann for proposing the question in which this paper found its original inspiration,  Ef Sofos for suggesting the particular problem addressed in this paper during a talk  given by the author at the Max Planck Institute for Mathematics in Bonn,  and Pieter Moree for the invitation to give that talk. In particular, the author is grateful to Igor Shparlinski for pointing out the bound \eqref{Shparli}. 
Part of the work was supported by the European Research Council under the European Union's Seventh Framework Programme (FP/2007--2013)/ERC Grant agreement no. 335220 --- AQSER. The work was completed during a stay at the Max Planck Institute for Mathematics. The author would like to acknowledge the hospitality of the staff and the inspiring atmosphere.


\begin{thebibliography}{99}
\bibitem{And} G. E. Andrews, \textit{The theory of partitions}, Cambridge Mathematical Library. Cambridge University Press,  1998.
\bibitem{AAR} G. E. Andrews, R. Askey and R. Roy, \textit{Special functions}, Encyclopedia of Mathematics and its Applications 71. Cambridge University Press, 1999.
\bibitem{Apostol} T. M. Apostol, \textit{Introduction to Analytic Number Theory}, Undergraduate Texts in Mathematics, Springer-Verlag, New York-Heidelberg, 1976.
\bibitem{Apostol2} T. M. Apostol, \textit{Modular Functions and Dirichlet Series in Number Theory}, Second edition, Graduate Texts in Mathematics 41. Springer-Verlag, New York, 1990.
\bibitem{BS} W. Banks and I. E. Shparlinski, On Gauss sums and the evaluation of Stechkin's constant, \textit{Math. Comp.} \textbf{85} (2015), no. 301, 2569--2581.
\bibitem{BM} K. Bringmann and K. Mahlburg, Transformation laws and asymptotics for nonmodular products, unpublished preprint.
\bibitem{AC} A. Ciolan, Asymptotics and inequalities for partitions into squares, \textit{Int. J. Number Theory} \textbf{16} (2020), no. 1, 121--143.
\bibitem{CP} T. Cochrane and C. Pinner, Explicit bounds on monomial and binomial exponential sums, \textit{Q.
J. Math.} \textbf{62} (2011), no. 2, 323--349.
\bibitem{CE} H. Cohn and N. Elkies, New upper bounds on sphere packings. I, \textit{Ann. of Math.} (2) \textbf{157}
(2003), no. 2, 689--714.
\bibitem{Glaisher} J. W. L. Glaisher; On formulae of verification in the partition of numbers, \textit{Proc. Royal Soc. London} \textbf{24} (1876), 250--259.
\bibitem{Liardet} P. Liardet and A. Thomas, Asymptotic formulas for partitions with
bounded multiplicity, in \textit{Applied algebra and number theory. Essays in honor of Harald Niederreiter on the occasion of his 70th birthday}, Cambridge: Cambridge University Press (2014), 235--254.
\bibitem{Mein} G. Meinardus, Asymptotische aussagen \"uber Partitionen, \textit{Math. Z.} \textbf{59} (1954), 388--398.
\bibitem{RothSzek}K. F. Roth and G. Szekeres, Some asymptotic formulae in the theory of partitions, 
\textit{Q. J. Math. Oxford} \textbf{5} (1954), no. 2, 241--259.
\bibitem{Shparlinski} I. E. Shparlinski\u \i, Estimates for Gauss sums (English transl.), \textit{Math. Notes} \textbf{50} (1991), no. 1--2, 740--746 (1992).
\bibitem{Stech} S. B. Ste\v ckin, An estimate for Gaussian sums (Russian), \textit{Mat. Zametki} \textbf{17} (1975), no. 4,
579--588.
\bibitem{Titch} E. C. Titchmarsh, \textit{The theory of the Riemann zeta-function}, Second edition. Edited and with a preface by D. R. Heath-Brown, The Clarendon Press, Oxford University Press, New York, 1986.

\bibitem{Vaughan} R. C. Vaughan, \textit{The Hardy–Littlewood Method}, Second edition, Cambridge Tracts in Mathematics, Cambridge University Press, 1997.

\bibitem{WrightIII} E. M. Wright, Asymptotic partition formulae. III. Partitions into $k$-th powers, \textit{Acta Math.} \textbf{63} (1934), no. 1, 143--191.
\bibitem{Zhou} N. H. Zhou, Note on partitions into polynomials with number of parts in an arithmetic progression, available as preprint at \url{https://arxiv.org/abs/1909.13549}.

		
\end{thebibliography}
\end{document}